%% file: article-REVREV.tex
\RequirePackage[l2tabu,orthodox]{nag}	%LaTeXe orthodoxe ?

\documentclass[a4paper,12pt,english]{article}

	\usepackage{lmodern}				%Latin modern, police vectorielle

	\usepackage[T1]{fontenc}			%Occidental

	\usepackage[utf8x]{inputenc}		%UTF8

\input{packmin}

\input{commath}

	\usepackage{hyperref}
	\usepackage{cleveref}
	\usepackage{authblk}

\newcommand*{\p}{\ensuremath\mathbb{P}} % proba
\newcommand*{\E}{\ensuremath\mathbb{E}} % proba
\newcommand{\var}{\mathrm{Var}}

\theoremstyle{plain}
\newtheorem{tm}{Theorem}
\newtheorem{prop}[tm]{Proposition}
\newtheorem{lm}[tm]{Lemma}
\theoremstyle{definition}
\newtheorem{rem}[tm]{Remark}

\hypersetup{colorlinks=true,linkcolor=black,urlcolor=black,citecolor=black}

\begin{document}

\title{Large deviations at the transition for sums of Weibull-like random variables}

\author[1]{Fabien Brosset}
\author[2]{Thierry Klein}
\author[3]{Agn\`es Lagnoux}
\author[1]{Pierre Petit}

\affil[1]{Institut de Math\'ematiques de Toulouse; UMR5219. Universit\'e de Toulouse; CNRS. UT3, F-31062 Toulouse, France.}
\affil[2]{Institut de Math\'ematiques de Toulouse; UMR5219. Universit\'e de Toulouse; ENAC - Ecole Nationale de l'Aviation Civile , Universit\'e de Toulouse, France.}
\affil[3]{Institut de Math\'ematiques de Toulouse; UMR5219. Universit\'e de Toulouse; CNRS. UT2J, F-31058 Toulouse, France.}

\maketitle

\begin{abstract}
Deviation probabilities of the sum $S_n=X_1+\dots+X_n$ of independent and identically distributed real-valued random variables have been extensively investigated, in particular when $X_1$ is Weibull-like distributed, i.e.\ $\log \p(X\geqslant x) \sim -qx^{1-\epsilon}$ as $x\to \infty$. For instance, A.V.\ Nagaev formulated exact asymptotic results for $\p(S_n>x_n)$ when $x_n > n^{1/2}$ (\emph{see}, \cite{Nagaev69-1,Nagaev69-2}). In this paper, we derive rough asymptotic results (at logarithmic scale) with shorter proofs relying on classical tools of large deviation theory and giving an explicit formula for the rate function at the transition $x_n = \Theta(n^{1/(1+\epsilon)})$.
\end{abstract}



%\tableofcontents

\section{Introduction} \label{sec:intro}

Moderate and large deviations of the sum of independent and identically distributed (i.i.d.) real-valued random variables have been investigated since the beginning of the 20th century. 
Kinchin \cite{Kinchin29} in 1929 was the first to give a result on large deviations of i.i.d. Bernoulli distributed random variables. In 1933, Smirnov  \cite{Smirnov33} improved this result and in 1938 Cramér \cite{Cramer38} gave a generalization to i.i.d.\ random variables satisfying the eponymous Cramér's condition which requires the Laplace transform of the common distribution of the random variables to be finite in a neighborhood of zero. Cramér's result was extended by Feller \cite{Feller43} to sequences of not necessarily identically distributed random variables under restrictive conditions
(Feller considered only random variables taking values in bounded intervals),
thus Cramér's result does not follow from Feller's result. A strengthening of Cramér's theorem was given by Petrov in \cite{Petrov54} together with a generalization to the case
of non-identically distributed random variables. Improvements of Petrov's result can be found in \cite{petrov2008large}.
Deviations for sums of heavy-tailed i.i.d.\ random variables were studied by several authors: an early result appears in \cite{Linnik61} and more recent references are \cite{borovkov2000large,borovkov2008asymptotic,denisov2008large,mikosch1998large}.

%The case where the tail decreases faster than all power functions (but not enough for Cramér's condition to be satisfied) has been considered by Petrov \cite{Petrov54} and by S.V.\ Nagaev \cite{Nagaev62}.
In \cite{Nagaev69-1,Nagaev69-2}, A.V.\ Nagaev  studied the case where the commom distribution of the i.i.d.\ random variables is absolutely continuous with respect to the Lebesgue measure with density  $p(t)\sim e^{-\abs{t}^{1-\epsilon}}$ as $\abs{t}$ tends to infinity, with $\epsilon \in (0,1)$. He distinguished five exact-asymptotics results corresponding to five types of deviation speeds. In \cite{Nagaev_1979_AnnProb}, S.V.\ Nagaev generalizes to the case where the tail writes as $e^{-t^{1-\epsilon}L(t)}$, where $\epsilon \in (0,1)$ and $L$ is a suitably slowly varying function at infinity. Such results can also be found in \cite{borovkov2000large,borovkov2008asymptotic}.

\medskip

Now, let us present the setting of this article. Let $\epsilon \in (0,1)$ and let $X$ be a real-valued random variable verifying: there exists $q>0$ such that
\begin{equation}
\label{eq:behav_X}
\log \p(X\geqslant x) \sim -qx^{1-\epsilon} \quad \text{as $x\to \infty$.}
\end{equation}
Such a random variable $X$ is often called a Weibull-like (or semiexponential, or stretched exponential) random variable.
One particular example is that of \cite{Nagaev69-1,Nagaev69-2} where $X$ has a density $p(x)\sim e^{-x^{1-\epsilon}}$. 
Moreover, unlike in \cite{Nagaev69-1,Nagaev69-2}, this unilateral assumption is motivated by the fact that we focus on upper deviations of the sum. Observe that \eqref{eq:behav_X} implies that the Laplace transform of $X$ is not defined on the right side of zero. Nevertheless, all moments of $X_+\defeq \max(X,0)$ are finite. A weaker assumption on the left tail is required:
\begin{equation}
\label{hyp1}
\exists \gamma>0 \quad \rho \defeq \E[|X|^{2+\gamma}] < \infty .
\end{equation}
We assume that $X$ is centered ($\E[X] = 0$) and denote by $\sigma$ the standard deviation of $X$ ($\var(X) = \sigma^2$). For all $n \in \N^*$, let $X_1$, $X_2$, ..., $X_n$ be i.i.d.\ copies of $X$. We set $S_n=X_1+\dots+X_n$. %where $x\defeq x_n$ is a sequence converging to infinity as $n$ goes to infinity
In this paper, we are interested in the asymptotic behavior of $\log \p(S_n\geqslant x_n)$ for any positive sequence $x_n \gg n^{1/2}$. Not only does the logarithmic scale allow us to use the modern theory of large deviations and provide simpler proofs than in \cite{borovkov2000large, borovkov2008asymptotic, Nagaev_1979_AnnProb, Nagaev69-1, Nagaev69-2}, but we also obtain more explicit results. According to the asymptotics of $x_n$, only three logarithmic asymptotic ranges appear. First, the Gaussian range: when $x_n \ll n^{1/(1+\epsilon)}$, $\log \p(S_n\geqslant x_n)\sim \log(1-\phi(\sigma^{-1} n^{-1/2} x_n))$, $\phi$ being the distribution function of the standard Gaussian law. Next, the domain of validity of the maximal jump principle: when $x_n \gg n^{1/(1+\epsilon)}$, $\log \p(S_n\geqslant x_n)\sim \log\p(\max(X_1,\ldots,X_n)\geqslant x_n)$. Finally, the transition ($x_n = \Theta(n^{1/(1+\epsilon)})$) appears to be an interpolation between the Gaussian range and the maximal jump one.

Logarithmic asymptotics were also considered in \cite{lehtomaa2017large} for a wider class of distributions than in the present paper. Nevertheless, the setting was restricted to the particular sequence $x_n=n$ (that lies in the maximal jump range). In \cite{eichelsbacher2003moderate}, the authors gave a necessary and sufficient condition on the logarithmic tails of the sum of i.i.d.\ real-valued random variables to satisfy a large deviation principle which covers the Gaussian range. In \cite{arcones2002large}, Arcones proceeded analogously and covered the maximal jump range for symmetric random variables. %Observe that, in the two previous papers, additional restrictions on $x_n$ were required. 
In \cite{gantert2000maximum}, the author studied a more general case of Weibull-like upper tails with a slowly varying function $L$, at a particular speed of the maximal jump range: $x_n=n^{1/(1-\epsilon)}$.

The transition at $x_n = \Theta(n^{1/(1+\epsilon)})$ is not considered in \cite{borovkov2000large,borovkov2008asymptotic}. It is treated in \cite{Nagaev_1979_AnnProb} and in \cite[Theorems 2 and 4]{Nagaev69-1,Nagaev69-2}. Nevertheless, the rate function is given through non explicit formulae and hence is difficult to interpret. The main contribution of this work is to provide an explicit formula for the rate function at the transition. Moreover, we provide probabilistic proofs which apply both to  the Gaussian range and to the transition. 

The paper is organized as follows. In Section \ref{sec:main}, we recall two known results (Theorems \ref{tm-gaussian} and \ref{tm-max}) and state the main theorem (Theorem \ref{tm-inter}). Section \ref{sec:preliminary} is devoted to preliminary results. In particular, we recall a unilateral version of Gärtner-Ellis theorem inspired from \cite{PS75} (Theorem \ref{th:plachky}) and establish a unilateral version of the contraction principle for a sum (Proposition \ref{principe-contraction-couple}), which has its own interest and which we did not find in the literature. The proof of Theorem \ref{tm-inter} can be found in Section \ref{sec:proof_13}. On the way, we prove Theorem \ref{tm-gaussian}. And, to be self-contained, we give in Section \ref{sec:proof_2} a short proof of Theorem \ref{tm-max} which is new, up to our knowledge.

\section{Main result} \label{sec:main}

In this section, we summarize all regimes of deviations for the sum $S_n$ defined in Section \ref{sec:intro}. The two following results are known (see, e.g.,\ \cite{eichelsbacher2003moderate} and \cite{borovkov2000large}).

\begin{tm}[Gaussian range] \label{tm-gaussian}
For $n^{1/2} \ll x_n \ll n^{1/(1+\epsilon)}$, we have: 
\[
\lim_{n \to \infty} \frac{n}{x_n^2} \log \p(S_n\geqslant x_n)= -\frac{1}{2\sigma^2} .
\]
\end{tm}

\begin{tm}[Maximal jump range] \label{tm-max}
For $x_n \gg n^{1/(1+\epsilon)}$, setting $M_n:=\max(X_1,\ldots,X_n)$,
\[
\lim_{n \to \infty} \frac{1}{x_n^{1-\epsilon}} \log \p(S_n\geqslant x_n)=\lim_{n \to \infty} \frac{1}{x_n^{1-\epsilon}} \log\p(M_n\geqslant x_n)=\lim_{n \to \infty} \frac{1}{x_n^{1-\epsilon}}\log \p(X \geqslant x_n) = -q .
\]
\end{tm}

The Gaussian range occurs when all summands contribute to the deviations of $S_n$ in the sense that $\log \p(S_n\geqslant x_n)\sim \log \p(S_n\geqslant x_n,\ \forall i\in \llbracket 1,n \rrbracket \quad X_i < x_n^{\epsilon})$. In the maximal jump range, the main contribution of the deviations of $S_n$ is due to one summand, meaning that $\log \p(S_n\geqslant x_n)\sim \log \p(X \geqslant x_n)$.

\medskip

Now we turn to the main contribution of this paper: we estimate the deviations of $S_n$ at the transition $x_n = \Theta(n^{1/(1+\epsilon)})$ and provide an explicit formula for the rate function. Notice that the sequence $n^{1/(1+\epsilon)}$ is the solution (up to a scalar factor) of the following equation in $x_n$: $x_n^2/n = x_n^{1-\epsilon}$, equalizing the speeds of of the deviation results obtained in the Gaussian range and in the maximal jump range. It appears that the behavior at the transition is a trade-off between the Gaussian range and the maximal jump range driven by the contraction principle for the distributions $\mathcal{L}(S_{n-1}\ |\ \forall i\ X_i<x_n^\epsilon) \otimes \mathcal{L}(X_n \ |\ X_n \geqslant x_n^\epsilon)$ and the function sum.

\begin{tm}[Transition] \label{tm-inter}
For all $C>0$ and $x_n=Cn^{1/(1+\epsilon)}$,
\[
\lim_{n \to \infty} \frac{n}{x_n^2} \log \p(S_n\geqslant x_n)
 = - \inf_{0 \leqslant t \leqslant 1} \left\{\frac{q (1-t)^{1-\epsilon}}{C^{1+\epsilon}} + \frac{t^2 }{2\sigma^2} \right\}
 \eqdef - J(C) .
\]
\end{tm}

%\begin{tm}[Gaussian range]\label{tm-gaussian} For $n^{1/2} \ll x_n \ll n^{1/(1+\epsilon)}$, we have: 
%\[
%\lim_{n \to \infty} \frac{n}{x_n^2} \log \p(S_n\geqslant x_n)= -\frac{1}{2\sigma^2}.
%\]
%\end{tm}

%\begin{tm}[Transition] \label{tm-inter} For $C>0$ and $x_n=Cn^{1/(1+\epsilon)}$,
%\[
%\lim_{n \to \infty} \frac{1}{x_n^{1-\epsilon}} \log \p(S_n\geqslant x_n)
% = - \inf_{0 \leqslant t \leqslant 1} \left\{ (1-t)^{1-\epsilon} + \frac{t^2 C^{1+\epsilon}}{2\sigma^2} \right\}
% \eqdef - J(C) .
%\]
%\end{tm}

Let us give a somewhat more explicit expression for the rate function $J$. Let $f(t)=q(1-t)^{1-\epsilon}/C^{1+\epsilon} + t^2/(2\sigma^2)$. An easy computation shows that, if $C \leqslant C'_{\epsilon} \defeq (1+\epsilon)((1-\epsilon)q\sigma^2\epsilon^{-\epsilon})^{1/(1+\epsilon)}$, then $f$ is decreasing and its minimum $1/(2\sigma^2)$ is attained at $t=1$. If $C>C'_{\epsilon}$, then $f$ has two local minima, at $1$ and at $t(C)$: the latter corresponds to the smallest of the two roots in $[0,1]$ of $f'(t)=0$, equation equivalent to
\begin{align*} %\label{eq:nag_6}
t(1-t)^{\epsilon}=\frac{(1-\epsilon)q\sigma^2}{C^{1+\epsilon}} .
\end{align*}
If $C'_\epsilon < C \leqslant C_\epsilon \defeq (1+\epsilon)(q\sigma^2(2\epsilon)^{-\epsilon})^{1/(1+\epsilon)}$, then $f(t(C)) \geqslant f(1)$. And, if $C > C_\epsilon$, then $f(t(C)) < f(1)$. As a consequence, for all $C> 0$,
\begin{align*} %\label{def:fct_taux}
J(C) = \begin{cases}
\frac{1}{2\sigma^2} & \text{if  $C\leqslant C_{\epsilon}$,}\\ 
\frac{q(1-t(C))^{1-\epsilon}}{C^{1+\epsilon}}+\frac{t(C)^2}{2\sigma^2} & \text{if  $C>C_{\epsilon}$.}
\end{cases}
\end{align*}

As a consequence, we see that the transition interpolates between the Gaussian range and the maximal jump one. First, when $x_n=Cn^{1/(1+\epsilon)}$, the asymptotics of the Gaussian range coincide with the one of the transition for $C\leqslant C_{\epsilon}$.
Moreover, $t(C_{\epsilon})=(1-\epsilon)/(1+\epsilon)$ and one can check that $-1/(2\sigma^2) = -q(1-t(C_{\epsilon}))^{1-\epsilon}/C_{\epsilon}^{1+\epsilon} - t(C_{\epsilon})^2/(2\sigma^2)$. Finally, for $C>C_{\epsilon}$, by the definition of $t(C)$, we deduce that, as $C\to \infty$, $t(C)\to 0$ leading to $t(C){\sim}(1-\epsilon)q\sigma^2C^{-(1+\epsilon)}$. Consequently, $C^{1+\epsilon}J(C) \to 1$ as $C\to \infty$, and we recover the asymptotic of the maximal jump range (recall that, when $x_n=Cn^{1/(1+\epsilon)}$, $x_n^2/n=C^{1+\epsilon} x_n^{1-\epsilon}$).

%\begin{tm}[Maximal jump range] \label{tm-max} For $x_n \gg n^{1/(1+\epsilon)}$, setting $M_n:=\max(X_1,\ldots,X_n)$,
%\[
%\lim_{n \to \infty} \frac{1}{x_n^{1-\epsilon}} \log \p(S_n\geqslant x_n)=\lim_{n \to \infty} \frac{1}{x_n^{1-\epsilon}} \log\p(M_n\geqslant x_n)=\lim_{n \to \infty} \frac{1}{x_n^{1-\epsilon}}\log \p(X \geqslant x_n)=-1
%.\]
%\end{tm}

\medskip

In Section \ref{sec:proof_13}, we give a proof of Theorem \ref{tm-inter} which also encompasses Theorem \ref{tm-gaussian}. Before turning to this proof, we establish several intermediate results useful in the sequel.

\section{Preliminary results}\label{sec:preliminary}

First, we present a classical result, known as the principle of the largest term, that will allow us to consider the maximum of several quantities rather than their sum. The proof is standard (see, e.g., \cite[Lemma 1.2.15]{DZ98}).

\begin{lm}[Principle of the largest term] \label{lm-magique}
Let $(v_n)_{n\geqslant 0}$ be a positive sequence diverging to $\infty$, $N$ be a positive integer, and, for $i=1,\ldots,N$, $(a_{n,i})_{n\geqslant 0}$ be a sequence of non-negative numbers. Then, 
\[
\varlimsup_{n \to \infty} \frac{1}{v_n} \log \left(\sum_{i=1}^N a_{n,i}\right)=\max_{i=1,\ldots,N}\left( \varlimsup_{n \to \infty} \frac{1}{v_n} \log a_{n,i}\right) .
\]
\end{lm}

The next theorem is a unilateral version of Gärtner-Ellis theorem, which was proved in \cite{PS75}. Its proof is omitted to lighten the present paper.

\begin{tm}[Unilateral Gärtner-Ellis theorem] \label{th:plachky}
Let $(Y_n)_{n\geqslant 0}$ be a sequence of real random variables and a positive sequence $(v_n)_{n\geqslant 0}$ diverging to $\infty$. Suppose that there exists a differentiable function $\Lambda$ defined on $\R_+$ such that $\Lambda'$ is a (increasing) bijective function from $\R_+$ to $\R_+$ and, for all $\lambda\geqslant 0$:
\[
\frac{1}{v_n} \log\E\bigl[ e^{v_n\lambda Y_n} \bigr] \xrightarrow[n\to\infty]{} \Lambda(\lambda) .
\]
Then, for all $c\geqslant 0$, 
\[
-\inf_{t>c} \Lambda^*(t) \leqslant \varliminf_{n\to\infty} \frac{1}{v_n}\log \p(Y_n>c) \leqslant \varlimsup_{n\to\infty} \frac{1}{v_n}\log \p(Y_n\geqslant c) \leqslant -\inf_{t\geqslant c} \Lambda^*(t),
\]
where, for all $t\geqslant 0$, $\Lambda^*(t)\defeq \sup \{ \lambda t -\Lambda(\lambda) \ ;\ \lambda \geqslant 0\}$.
\end{tm}

Now, we present a unilateral version of the contraction principle for a sequence of random variables in $\R^2$ with independent coordinates where the function considered is the sum of the coordinates. Observe that only unilateral assumptions are required. The proof of the upper bound uses the same kind of decomposition as in the proof of 
\cite[Lemma 4.3]{dyszewski2020maximum}.

\begin{prop}[Unilateral sum-contraction principle] \label{principe-contraction-couple} Let  $((Y_{n,1} ,Y_{n,2}))_{n\geqslant 0}$ be a sequence of $\R^2$-valued random variables such that, for each $n$, $Y_{n,1}$ and $Y_{n,2}$ are independent. Let $(v_n)_{n\geqslant 0}$ be a positive sequence diverging to $\infty$. For all $a \in \R$ and $i \in \{ 1, 2 \}$, let us define
\[
\underline{I}_i(a) = - \inf_{u<a} \varliminf_{n\to \infty} \frac{1}{v_n} \log\p(Y_{n,i}>u)
\quad \text{and} \quad
\overline{I}_i(a) = - \inf_{u<a} \varlimsup_{n\to \infty} \frac{1}{v_n} \log\p(Y_{n,i}>u) .
\] 
Assume that:\newline
{\bfseries\upshape (H)} for all $M > 0$, there exists $d > 0$ such that
\[
\varlimsup_{n\to \infty} \frac{1}{v_n} \log\p(Y_{n,1}>d,\ Y_{n,2}<-d) < -M
\quad \text{and} \quad
\varlimsup_{n\to \infty} \frac{1}{v_n} \log\p(Y_{n,1}<-d,\ Y_{n,2}>d) < -M .
\]
Then, for all $c\in\R$, one has
\[
-\inf_{t> c} \underline{I}(t)
\leqslant
\varliminf_{n\to \infty} \frac{1}{v_n} \log\p(Y_{n,1}+Y_{n,2}> c) \leqslant 
\varlimsup_{n\to \infty} \frac{1}{v_n} \log\p(Y_{n,1}+Y_{n,2}\geqslant c) \leqslant -\inf_{t\geqslant c} \overline{I}(t)
\]

where, for all $t \in \R$,
\[
\underline{I}(t)\defeq\inf_{\substack{ a,b \in \R \\ a+b=t}} \underline{I}_1(a)+\underline{I}_2(b)
\quad \text{and} \quad
\overline{I}(t)\defeq\inf_{\substack{ a,b \in \R \\ a+b=t}} \overline{I}_1(a)+\overline{I}_2(b) .
\]
Moreover $\underline{I}$ and $\overline{I}$ are nondecreasing functions.
\end{prop}

\begin{rem}
A sufficient condition for assumption \textbf{(H)} is: for $i \in \{ 1, 2 \}$,
\[
\lim_{a\rightarrow \infty} \overline{I}_i(a)=\infty .
\]
\end{rem}

\begin{proof}
Obviously, the functions $\underline I_1$, $\overline I_1$, $\underline I_2$, and $\overline I_2$ are nondecreasing.
Let us prove that $\underline I$ is nondecreasing, the proof for $\overline I$ being similar. Let $t_1 < t_2$, let $\eta > 0$, and let $a \in \R$ be such that $\underline I(t_2) \geqslant \underline I_1(a) + \underline I_2(t_2-a) - \eta$. Since $\underline I_2$ is nondecreasing, we have
\[
\underline I(t_1)
 \leqslant \underline I_1(a) + \underline I_2(t_1-a)
 \leqslant \underline I_1(a) + \underline I_2(t_2-a)
 \leqslant \underline I(t_2) + \eta ,
\]
which completes the proof of the monotony of $\underline I$, letting $\eta \to 0$. 

\medskip

\textit{Lower bound.} Let $c\in \R$, let $t>c$, and let $\delta > 0$ be such that $0<2\delta<t-c$. For all $(a,b)\in \R^2$ such that $a+b=t$, we have
\begin{align*}
\varliminf \frac{1}{v_n} \log \p(Y_{n,1}+Y_{n,2}>c)
 %& \geqslant  \varliminf \frac{1}{v_n} \log \p(Y_{n,1}>a-\delta) +   \varliminf \frac{1}{v_n} \log \p(Y_{n,2}>b-\delta)\\
 & \geqslant \varliminf \frac{1}{v_n} \log\p(Y_{n,1}>a-\delta) + \varliminf \frac{1}{v_n}  \log\p(Y_{n,2}>b-\delta) \\
 & \geqslant -\underline I_1(a)-\underline I_2(b) .
\end{align*}
Therefore,
\[
\varliminf \frac{1}{v_n}  \log \p(Y_{n,1}+Y_{n,2}>c)
 \geqslant \sup_{t>c} \sup_{\substack{ (a,b) \in \R^2 \\ a+b=t}} (-\underline I_1(a)-\underline I_2(b))
 = - \inf_{t> c} \underline I(t) .
\]

\textit{Upper bound.} Let $c \in \R$ and let $M > 0$. Let $d > 0$ be given by assumption \textbf{(H)}. Define
\[
Z = \{ (a,b) \in \intervallefo{-d}{\infty}^2 \ ;\ a+b \geqslant c \}
\quad \text{and} \quad
K = \{ (a,b) \in \intervallefo{-d}{\infty}^2 \ ;\ a+b = c \} .
\]
Write
\begin{align*}
\p(Y_{n,1}+Y_{n,2} \geqslant c)
 & \leqslant \p(Y_{n,1}>d,\ Y_{n,2}<-d) + \p(Y_{n,1}<-d,\ Y_{n,2}>d) + \p\bigl((Y_{n,1},Y_{n,2}) \in Z \bigr) \\
 & \eqdef Q_{n,1}+Q_{n,2}+Q_{n,3} .
\end{align*}
By assumption,
\[
\varlimsup \frac{1}{v_n}  \log(Q_{n,1}) < -M
\quad \text{and} \quad
\varlimsup \frac{1}{v_n}  \log(Q_{n,2}) < -M .
\]
Let us estimate $\varlimsup v_n^{-1}  \log Q_{n,3}$. For all $(a,b) \in K$,
\begin{align*}
- \inf_{\substack{u<a \\ v<b}}
 & \varlimsup\frac{1}{v_n} \log \p\bigl( Y_{n,1} > u,\ Y_{n,2} > v \bigr) \\
 %& \geqslant - \inf_{\substack{u<a \\ v<b}} \left[ \varlimsup\frac{1}{v_n} \log \p(Y_{n,1}>u) + \varlimsup\frac{1}{v_n} \log \p(Y_{n,2}>v) \right]\\
 & \geqslant - \inf_{u<a} \varlimsup \frac{1}{v_n}  \log \p(Y_{n,1}>u) - \inf_{v<b} \varlimsup \frac{1}{v_n}  \log \p(Y_{n,2}>v) \\
 & = \overline{I}_1(a) + \overline{I}_2(b) .
\end{align*}
Defining $\theta^{[\delta]} \defeq \min(\theta-\delta,\delta^{-1})$ for all $\delta>0$ and for all $\theta \in (-\infty,\infty]$, there exists $u_a<a$ and $v_b<b$ such that
\begin{equation} \label{upper-bound-delta}
- \varlimsup \frac{1}{v_n}  \log \p\bigl( Y_{n,1} > u_a,\ Y_{n,2} > v_b \bigr)
 \geqslant (\overline{I}_1(a) + \overline{I}_2(b))^{[\delta]} .
\end{equation}
From the cover $((u_a,\infty) \times (v_b,\infty))_{(a,b) \in K}$ of the compact subset $K$, we can extract a finite subcover $((u_{a_i},\infty) \times (v_{b_i},\infty))_{1 \leqslant i \leqslant p}$. Since
\[
Z \subset \bigcup_{i=1}^p (u_{a_i},\infty) \times (v_{b_i},\infty),
\]
we obtain, thanks to Lemma \ref{lm-magique} and \eqref{upper-bound-delta},
\begin{align*}
\varlimsup \frac{1}{v_n}  \log Q_3
 %& \leqslant \varlimsup\frac{1}{v_n}  \log \p\bigg((Y_{n,1},Y_{n,2})\in \bigcup_{i=1}^p ]u_{a_i},\infty[\times]v_{b_i},\infty[\bigg) \\
 & \leqslant \varlimsup \frac{1}{v_n}  \log \sum_{i=1}^p \p\bigl(Y_{n,1} > u_{a_i},\ Y_{n,2} > v_{b_i} \bigr) \\
 & = \max_{1 \leqslant i \leqslant p} \bigl\{ \varlimsup \frac{1}{v_n}  \log \p\bigl(Y_{n,1} > u_{a_i},\ Y_{n,2} > v_{b_i}\bigr) \bigr\} \\
 & \leqslant \max_{1 \leqslant i \leqslant p} \bigl\{ - (\overline{I}_1(a_i) + \overline{I}_2(b_i))^{[\delta]} \bigr\}\\
 %& = -\min_{1 \leqslant i \leqslant p}\left\{(I_1(a_i)+I_2(b_i))^{[\delta]}\right\}\\
 & \leqslant -\inf_{\substack{ (a,b) \in \R^2 \\ a+b=c}}\bigl\{ (\overline{I}_1(a)+\overline{I}_2(b))^{[\delta]}\bigr\} .
\end{align*}
Letting $\delta \to 0$ and using the definition of $\overline{I}$, we deduce that
\begin{align*} %\label{upper-bound-P3}
\varlimsup \frac{1}{v_n}  \log Q_3
 & \leqslant -\inf_{\substack{ (a,b) \in \R^2 \\ a+b=c}}(\overline{I}_1(a)+\overline{I}_2(b)) 
 = -\overline{I}(c)
 = -\inf_{t\geqslant c} \overline{I}(t) .
\end{align*}
Letting $M \to \infty$, we get the desired upper bound.
\end{proof}

\section{Proof of Theorems \ref{tm-gaussian} and \ref{tm-inter}}\label{sec:proof_13}

From now on, all non explicitly mentioned asymptotics are taken as $n \to \infty$. Replacing $X$ by $q^{-1/(1-\epsilon)} X$, we may suppose without loss of generality that
\begin{equation}
\label{eq:behav_X1}
\log \p(X\geqslant x) \sim -x^{1-\epsilon} \quad \text{as $x\to \infty$.}
\end{equation}

The conclusions of Theorem \ref{tm-gaussian} and \ref{tm-inter} follow from Lemmas \ref{lem:eq-gauss-Pin0}, \ref{lem:eq-gauss-Pin1_inter}, \ref{lem:maj_pinm} below, and the principle of the largest term (Lemma \ref{lm-magique}).

\subsection{Principal estimates}

By \eqref{eq:behav_X1}, the Laplace transform $\Lambda_X$ of $X$ is not defined at the right of zero. In order to use the standard exponential Chebyshev inequality anyway, we introduce the following decomposition:
\begin{align*}
\Prob(S_n \geqslant x_n) & = \sum_{m=0}^n \binom{n}{m} \Pi_{n,m}(x_n) 
\end{align*}
where, for all $m \in \intervallentff{0}{n}$ and for all $a \geqslant 0$,
\[
\Pi_{n,m}(a) \defeq \Prob(S_n \geqslant a,\ \forall i \in \intervallentff{1}{m} \quad  X_i\geqslant x_n^\epsilon  ,\ \forall i \in \intervallentff{m+1}{n} \quad X_i < x_n^\epsilon).
\]

Note that the only relevant truncation is at $x_n^\epsilon$ (and not at $x_n$ as in \cite{Nagaev69-1,Nagaev69-2}). The asymptotics we want to prove are given by Lemmas \ref{lem:eq-gauss-Pin0} and \ref{lem:eq-gauss-Pin1_inter}, the proofs of which rely on the unilateral version of Gärtner-Ellis theorem (Theorem \ref{th:plachky}) and on the unilateral sum-contraction principle (Proposition \ref{principe-contraction-couple}).

\begin{lm} \label{lem:eq-gauss-Pin0}
Let $C > 0$. If $n^{1/2} \ll x_n  \leqslant C n^{1/(1+\epsilon)}$ and $t > 0$, then
\begin{equation} \label{eq-gauss-Pin0}
\lim \frac{n}{x_n^2} \log\Pi_{n,0}(tx_n) = -\frac{t^2}{2\sigma^2} .
\end{equation}
\end{lm}

\begin{proof}
Let us introduce $\overline{X}$ with distribution $\mathcal{L}(X\ |\ X < x_n ^\epsilon)$. For all $n \in \N^*$, let $\overline{X}_1$, $\overline{X}_2$, ..., $\overline{X}_n$ be i.i.d.\ copies of $\overline{X}$ and let $\overline{S}_n=\overline{X}_1+\dots+\overline{X}_n$, so that
\begin{align*}
\Pi_{n,0}(tx_n)
 & =\p(S_n\geqslant tx_n \,,\, X_1,\ldots, X_n<x_n ^{\epsilon})
 = \p(\overline{S}_n \geqslant tx_n ) \p(X < x_n^{\epsilon})^n
 \sim \p(\overline{S}_n \geqslant tx_n ) ,
\end{align*}
by \eqref{eq:behav_X1}.
We want to apply Theorem \ref{th:plachky} to the random variables $\overline{S}_n/x_n $ with $v_n =x_n ^2/n$. For $u>0$,
\begin{align}
%\Lambda_n(u)
%&\defeq\Lambda_{\overline{S}_n^{x_n ^{\epsilon}}/x_n }(u) =
\frac{n}{x_n ^2}\log \E \Bigl[e^{u\frac{x_n ^2}{n}\frac{\overline{S}_n}{x_n }}\Bigr]
=\frac{n^2}{x_n ^2}\log \E \left[e^{\frac{u x_n  X}{n}} \indic_{X < x_n ^\epsilon} \right] - \frac{n^2}{x_n ^2}\log \p(X < x_n ^\epsilon) \label{eq:gaussian_1} .
\end{align}
The second term in the right side of the above equation goes to 0 as $n\to \infty$ since $\log\p(X<x_n ^{\epsilon})\sim -\p(X>x_n ^{\epsilon})= O(e^{-x_n ^{\epsilon(1-\epsilon)}/2})$, by \eqref{eq:behav_X1}. As for the first term, if $y < x_n ^{\epsilon}$, then $x_n y/n \leqslant {x_n ^{1+\epsilon}}/{n} \leqslant C^{1+\epsilon}$. Now, up to changing $\gamma$ in $\gamma \wedge 1$, \eqref{eq:behav_X} is true for some $\gamma \in \intervalleof{0}{1}$ and there exists $c > 0$ such that, for all $s \leqslant C^{1+\epsilon}$, $|e^s - (1 + s + s^2/2)| \leqslant c|s|^{2+\gamma}$. Hence,
\begin{align}
 & \left| \E \left[e^{\frac{u x_n  X}{n}} \indic_{X < x_n ^\epsilon} \right] - e^{\frac{u^2 x_n ^2 \sigma^2}{2n^2}} \right| \nonumber \\
\leqslant & \left| \E \left[e^{\frac{u x_n  X}{n}} \indic_{X < x_n ^\epsilon} \right] - \E \left[\left(1+\frac{ux_n X}{n}+\frac{u^2x_n ^2X^2}{2n^2}\right) \indic_{X < x_n ^\epsilon} \right] \right| \nonumber \\
 & \hspace{2cm} + \left| \E \left[\left(1+\frac{ux_n X}{n}+\frac{u^2x_n ^2X^2}{2n^2}\right) \indic_{X < x_n ^\epsilon} \right] - \left( 1 + \frac{u^2 x_n ^2 \sigma^2}{2n^2} \right) \right| \nonumber \\
 & \hspace{2cm} + \left| \left( 1 + \frac{u^2 x_n ^2 \sigma^2}{2n^2} \right) - e^{\frac{u^2 x_n ^2 \sigma^2}{2n^2}} \right| \nonumber \\
\leqslant & \mathop{} c \rho \left( \frac{ux_n }{n} \right)^{2+\gamma} + \left| \E \left[\left(1+\frac{ux_n X}{n}+\frac{u^2x_n ^2X^2}{2n^2}\right) \indic_{X \geqslant x_n ^\epsilon} \right] \right| + o\left( \frac{x_n^2}{n^2} \right) . \label{eq:gaussian_2}
\end{align}
For $n$ large enough, applying Hölder's inequality,
\begin{align}
\left| \E \left[\left(1+\frac{ux_n X}{n}+\frac{u^2x_n ^2X^2}{2n^2}\right) \indic_{X \geqslant x_n ^\epsilon} \right] \right|
 & \leqslant \E[X^2 \indic_{X \geqslant x_n ^\epsilon}] \nonumber \\
 & \leqslant \E[X^{2+\gamma}]^{2/(2+\gamma)} \p(X \geqslant x_n ^\epsilon)^{\gamma/(2+\gamma)} \nonumber \\
 & = o\left( \frac{x_n ^2}{n^2} \right) \label{eq:gaussian_3} ,
\end{align}
by \eqref{eq:behav_X1}. Combining \eqref{eq:gaussian_1}, \eqref{eq:gaussian_2}, and \eqref{eq:gaussian_3}, we get
\[
\frac{n}{x_n ^2}\log \E \Bigl[e^{u\frac{x_n ^2}{n}\frac{\overline{S}_n}{x_n }}\Bigr] \to \frac{u^2 \sigma^2}{2} \defeq \Lambda(u) .
\]
Since $\Lambda^*(t) = t/{2\sigma^2}$, \eqref{eq-gauss-Pin0} stems from Theorem \ref{th:plachky}.
\end{proof}

\begin{lm}\label{lem:eq-gauss-Pin1_inter}
Let $C>0$. If $x_n=Cn^{1/(1+\epsilon)}$, then
\begin{align*}
\lim \frac{n}{x_n^2}\log \Pi_{n,1}(x_n) = -J(C).
\end{align*}
\end{lm}

\begin{proof}%[Proof of Lemma \ref{lem:eq-gauss-Pin1_inter}]
Recall that, for $n \in \N^*$,
\begin{align*}
\Pi_{n,1}(x_n)
&= \p(S_n\geqslant x_n \,|\,  X_1\geqslant  x_n ^{\epsilon} \,,\, X_2,\ldots,X_n< x_n ^{\epsilon})\p(  X \geqslant  x_n ^{\epsilon} )\p(X < x_n^{\epsilon})^{n-1}.
\end{align*}
Using \eqref{eq:behav_X1},
it suffices to prove that
\[
\frac{n}{x_n^2} \log \p(S_n\geqslant x_n \ | X_1 \geqslant  x_n ^{\epsilon} ,\ X_2, \dots, X_n < x_n ^{\epsilon}) \to -J(C) .
\]
To do so, we apply the contraction principle of Proposition \ref{principe-contraction-couple} to $(Y_{n,1}, Y_{n,2})$ with
\[
\mathcal{L}(Y_{n,1}) = \mathcal{L}(x_n ^{-1} X_1\ |\ X_1  \geqslant  x_n ^{\epsilon} )
\quad \text{and} \quad
\mathcal{L}(Y_{n,2}) = \mathcal{L}(x_n ^{-1}(X_2+\dots+X_n)\ |\ X_2,\dots,X_n < x_n ^{\epsilon}) ,
\]
and $v_n=x_n ^{2}/n$. First, one has obviously, $\p( X \geqslant ux_n  \ |\  X  \geqslant  x_n ^{\epsilon}) = 1$ for $u \leqslant 0$ and 
$\p(X \geqslant ux_n  \ |\  X  \geqslant  x_n ^{\epsilon} ) = 0$ for $u \geqslant 1$. In addition, for $u \in (0,1)$, using \eqref{eq:behav_X1},
\begin{align*}
\log \p(X \geqslant ux_n  \ |\  X  \geqslant  x_n ^{\epsilon} )
 & \sim -(ux_n )^{1-\epsilon} .
\end{align*}
Using the notation of Proposition \ref{principe-contraction-couple}, 
it follows that $\underline{I}_1(a) = \overline{I}_1(a) = I_1(a)$, where
\begin{equation}
\label{I1}
I_1(a)
 = \sup_{u<a}\left\{ - \lim \frac{n}{x_n^2}\log \p(X \geqslant ux_n \ |\ X \geqslant  x_n ^{\epsilon} ) \right\}
 = \begin{cases}
0 & \text{if $a< 0$,} \\
C^{-(1+\epsilon)}a^{1-\epsilon} & \text{if $a\in [0,1]$,} \\
\infty & \text{if $a>1$.}
\end{cases}
\end{equation} 
Moreover, for all $u>0$,
\[
\frac{n}{x_n^2} \log \p(X_2+\ldots+X_n \geqslant ux_n \ |\ X_2, \dots, X_n<x_n ^{\epsilon})
 = \frac{n}{x_n^2} \log \Pi_{n-1,0}(ux_n )
 \to -\frac{u^2 }{2\sigma^2} ,
\] 
by Lemma \ref{lem:eq-gauss-Pin0}. Thus, we have $I_2(b)=b^2/(2\sigma^2)$ for all 
$b>0$ and, since $I_2$ is a nondecreasing and nonnegative function, we get $I_2(b) = 0$ for all $b \leqslant 0$.
This, together with \eqref{I1}, leads to: for all $t\in \R$,
\[
I(t)
 = \inf_{\substack{ a+b=t \\ (a,b)\in \R^2}} \{ I_1(a) + I_2(b) \}
 = \inf_{t-1 \leqslant b \leqslant t} \{ I_1(t-b) + I_2(b) \}
 = \inf_{t-1 \leqslant b \leqslant t} \left\{ \frac{(t-b)^{1-\epsilon}}{C^{1+\epsilon}} + \frac{b^2 }{2\sigma^2} \right\} ,
\]
since $b < t-1$ entails $I_1(t-b) = \infty$ and $b > t$ entails $I_1(t-b)+I_2(b) > I_1(0) + I_2(t)$. It is a standard result (\emph{see, e.g.}, \cite[4.c.]{Moreau_1967_FonctionnellesConvexes}) that $I$ is upper semicontinuous. Since $I$ is also nondecreasing, $I$ is right continuous and we get
\[
%\label{infI}
\inf_{t \geqslant 1} I(t) = \inf_{t>1} I(t) = I(1) .
\]
Applying Proposition \ref{principe-contraction-couple}, this completes the proof.
\end{proof}

Notice that the very same argument shows that:
\begin{itemize}
\item if $x_n=Cn^{1/(1+\epsilon)}$, then, for all $m \geqslant 1$,
\begin{align*}
\lim \frac{n}{x_n^2}\log \Pi_{n,m}(x_n) = - J(C) ;
\end{align*}
\item if $x_n \ll n^{1/(1+\epsilon)}$, then, for all $m \geqslant 1$,
\begin{align*}
\lim \frac{n}{x_n^2}\log \Pi_{n,m}(x_n) = - \frac{1}{2 \sigma^2} .
\end{align*}
\end{itemize}
Our last step consist in proving that these estimates also hold for $\sum_{m=2}^n \binom{n}{m} \Pi_{n,m}(x_n)$ instead of $\Pi_{n,m}(x_n)$.

\subsection{Two uniform bounds}

\begin{lm} \label{lem:xi_small}
Fix a sequence $x_n \to \infty$. For all $\delta \in \intervalleoo{0}{1}$ and $M > 0$, there exists $n(\delta,M) \geqslant 1$ such that, for all $n \geqslant n(\delta,M)$, for all $m \in \intervallentff{0}{n}$, for all $u \in \intervalleff{0}{M n x_n^{-\epsilon}}$,
\[
\log \Prob(S_m \geqslant u,\ \forall i \in \intervallentff{1}{m} \quad X_i < x_n^\epsilon)
 \leqslant - \frac{(1-\delta) u^2}{2 n \sigma^2} .
\]
In particular, if $x_n \leqslant C n^{1/(1+\epsilon)}$, taking $M = C^{1+\epsilon}$, the bound holds for $u \in \intervalleff{0}{x_n}$.
\end{lm}

\begin{proof}
Using the fact that $\indic_{t \geqslant 0} \leqslant e^t$, for all $\lambda > 0$, 
\begin{align*}
\Prob(S_m \geqslant u,\ \forall i \in \intervallentff{1}{m} \quad X_i < x_n^\epsilon)
 & \leqslant e^{-\lambda u} \Espe[e^{\lambda X} \indic_{X < x_n^\epsilon}]^m .
\end{align*}
Up to changing $\gamma$ in $\gamma \wedge 1$, \eqref{eq:behav_X} is true for some $\gamma \in \intervalleof{0}{1}$ and there exists $c(M) > 0$ such that, for all $s \leqslant M \sigma^{-2}$, we have $e^s \leqslant 1 + s + s^2/2+c(M)|s|^{2+\gamma}$. Hence, for $\lambda = u (n\sigma^2)^{-1} \leqslant M \sigma^{-2} x_n^{-\epsilon}$,
\begin{align*}
\Espe[e^{\lambda X} \indic_{X < x_n^\epsilon}]
 & \leqslant 1 + \frac{\lambda^2 \sigma^2}{2} + c(M) \rho \lambda^{2+\gamma}
 \leqslant 1 + \frac{\lambda^2\sigma^2}{2}(1+\delta),
\end{align*}
as soon as $2 c(M) \rho \sigma^{-2} (M \sigma^{-2} x_n^{-\epsilon})^{\gamma} \leqslant \delta$, i.e.\ for $n \geqslant n(\delta,M)$. Thus, since $m \leqslant n$,
\begin{align*}
\log \Prob(S_m \geqslant u,\ \forall i \in \intervallentff{1}{m} \quad X_i < x_n^\epsilon)
 & \leqslant - \lambda u + \frac{n \lambda^2 \sigma^2}{2} (1+\delta) =  - \frac{(1-\delta) u^2}{2 n \sigma^2} . \qedhere
\end{align*}
\end{proof}

\begin{lm} \label{lem:xi_medium}
Fix a sequence $x_n \to \infty$. For all $\delta \in \intervalleoo{0}{1}$, there exists $n(\delta) \geqslant 1$ such that, for all $n \geqslant n(\delta)$, for all $m \geqslant 2$, for all $u \in \intervalleff{0}{x_n}$,
\[
\log \Prob(S_m \geqslant u,\ \forall i \in \intervallentff{1}{m} \quad X_i  \geqslant  x_n ^{\epsilon}) \leqslant - (1-\delta)\bigl( u^{1-\epsilon} + m (1-2^{-\epsilon}) x_n^{\epsilon(1-\epsilon)} \bigr) .
\]
\end{lm}

\begin{proof}
The result is trivial for $u < m x_n^\epsilon$. In the sequel, we suppose $ u  \geqslant  m x_n ^{\epsilon}$. Let $q' = 1-2\delta/3$ and $q'' = 1-\delta/3$, so that $1-\delta < q' < q'' < 1$. Choose $x(\delta) > 0$ such that, for all $x \geqslant x(\delta)$, $\log \Prob(X \geqslant x) \leqslant -q'' x^{1-\epsilon}$. One has:
\begin{align*}
\Prob(S_m \geqslant u,\ \forall i \in \intervallentff{1}{m} \quad X_i \geqslant  x_n^\epsilon)
 & \leqslant \Prob(S_m \geqslant u,\ \forall i \in \intervallentff{1}{m} \quad x_n^\epsilon \leqslant X_i < u) \\
 & \hspace{1cm} + \Prob(\exists i \in \intervallentff{1}{m}\quad X_i \geqslant  u,\ \forall i \in \intervallentff{1}{m} \quad X_i \geqslant x_n^\epsilon) .
\end{align*}
First,
\begin{align} \label{eq:xi_medium_1}
\Prob(\exists i  \in \intervallentff{1}{m} \quad X_i \geqslant u,\ \forall i \in \intervallentff{1}{m} \quad  X_i \geqslant x_n^\epsilon)
 & \leqslant m \Prob(X \geqslant u) \Prob(X \geqslant x_n^\epsilon)^{m-1} \nonumber \\
 & \leqslant m e^{-q'(u^{1-\epsilon} + (m-1) x_n^{\epsilon(1-\epsilon)})}
\end{align}
as soon as $x_n^\epsilon \geqslant x(\delta)$ (remember that $u \geqslant mx_n^\epsilon \geqslant x_n^\epsilon$), i.e.\ as soon as $n \geqslant n_1(\delta)$. Secondly, denoting by $a_i$ integers,
\begin{align*}
\Prob(S_m \geqslant u,\ \forall i \in \intervallentff{1}{m} \quad x_n^\epsilon \leqslant X_i < u)
 & = \int_{\forall i\ x_n^\epsilon \leqslant u_i < u} \indic_{u_1 + \dots + u_m \geqslant u} \prod_{i=1}^m \Prob(X \in du_i) \\
 & \leqslant \sum_{\forall i\ \ceil{x_n^\epsilon} \leqslant a_i \leqslant \ceil{u}} \indic_{a_1 + \dots + a_m \geqslant u} \prod_{i=1}^m \Prob(a_i-1 < X \leqslant a_i) \\
 & \leqslant \sum_{\forall i\ \ceil{x_n^\epsilon} \leqslant a_i \leqslant \ceil{u}} \indic_{a_1 + \dots + a_m \geqslant u} \prod_{i=1}^m e^{-q''(a_i-1)^{1-\epsilon}} \\
 & \leqslant \int_{\forall i\ x_n^\epsilon \leqslant u_i < u+2} \indic_{u_1 + \dots + u_m \geqslant u} \prod_{i=1}^m e^{-q''(u_i-2)^{1-\epsilon}} du_i \\
 & \leqslant \int_{\substack{\forall i\ x_n^\epsilon \leqslant u_i < u+2 \\ u_1 + \dots + u_m \geqslant u}} e^{-q'(u_1^{1-\epsilon} + \dots + u_m^{1-\epsilon})} du_1 \cdots du_m ,
\end{align*}
%
%
%
%
%\begin{align*}
%\Prob&(S_m \geqslant u,\ \forall i \in \intervallentff{1}{m} \quad  X_i  \geqslant  x_n ^{\epsilon})\\
% & = \int_{\substack{\forall i \in \intervallentff{1}{m},\\  u_i  \geqslant  x_n ^{\epsilon}}} \indic_{u_1 + \dots + u_m \geqslant u} \prod_{i=1}^m \Prob(X \in du_i) \\
% & \leqslant \sum_{\substack{\forall i \in \intervallentff{1}{m},\\  a_i \geqslant \ceil{x_n^\epsilon}}} \indic_{a_1 + \dots + a_m \geqslant u} \prod_{i=1}^m \Prob(a_i-1 < X \leqslant a_i) \\
% & \leqslant \sum_{\substack{\forall i \in \intervallentff{1}{m},\\  a_i \geqslant \ceil{x_n^\epsilon}}} \indic_{a_1 + \dots + a_m \geqslant u} \prod_{i=1}^m e^{-q''(a_i-1)^{1-\epsilon}} \\
% & \leqslant \int_{\substack{\forall i \in \intervallentff{1}{m},\\   u_i \geqslant  \ceil{x_n^\epsilon}}} \indic_{u_1 + \dots + u_m \geqslant u} \prod_{i=1}^m e^{-q''(u_i-2)^{1-\epsilon}} du_i \\
% & \leqslant \int_{\substack{\forall i \in \intervallentff{1}{m},\\ x_n^\epsilon \leqslant u_i < u \\ u_1 + \dots + u_m \geqslant u}} e^{-q'(u_1^{1-\epsilon} + \dots + u_m^{1-\epsilon})} du_1 \cdots du_m +\int_{\substack{ \forall i\ u_i \geqslant x_n^\epsilon,\\ \exists i\ u_i \geqslant  u }} e^{-q'(u_1^{1-\epsilon} + \dots + u_m^{1-\epsilon})} du_1 \cdots du_m,
%\end{align*}
as soon as $n$ is large enough ($n \geqslant n_2(\delta) \geqslant n_1(\delta)$) so that, for all $v \geqslant x_n^\epsilon$, $q''(v-2)^{1-\epsilon} \geqslant q'v^{1-\epsilon}$.
Now, the function $f \colon (u_1,\dots,u_m) \mapsto -q'(u_1^{1-\epsilon} + \dots + u_m^{1-\epsilon})$ is convex, so $f$ reaches its maximum on the domain of integration at the points where all the $u_i$ equal $x_n^\epsilon$, except one equal to $u-(m-1)x_n^\epsilon$. Therefore,
\begin{align*}
\Prob(S_m \geqslant u,\ & \forall i \in \intervallentff{1}{m} \quad  X_i \geqslant x_n^\epsilon)
 \leqslant (u+2)^m \exp\bigl[ -q'\bigl( (u-(m-1)x_n^\epsilon)^{1-\epsilon} + (m-1)x_n^{\epsilon(1-\epsilon)} \bigr) \bigr] .
\end{align*}
Let
\[
f(m,u) = (u-(m-1)x_n^\epsilon)^{1-\epsilon} + (m-1)x_n^{\epsilon(1-\epsilon)} - u^{1-\epsilon} .
\]
Since
\[
\frac{\partial f}{\partial u}(m,u) = (1-\epsilon) \biggl( \frac{1}{(u-(m-1)x_n^\epsilon)^\epsilon} - \frac{1}{u^\epsilon} \biggr) > 0
\]
and
\[
f(m,m x_n^\epsilon) = x_n^{\epsilon(1-\epsilon)} m(1-m^{-\epsilon}) \geqslant x_n^{\epsilon(1-\epsilon)} m(1-2^{-\epsilon}) ,
\]
we get
\begin{align} \label{eq:xi_medium_2}
\Prob(S_m \geqslant u,\ \forall i \in \intervallentff{1}{m} \quad x_n^\epsilon \leqslant X_i < u)
 & \leqslant (u+2)^m \exp\bigl[ -q'\bigl( u^{1-\epsilon} + m(1-2^{-\epsilon}) x_n^{\epsilon(1-\epsilon)} \bigr) \bigr] .
\end{align}
Finally, putting together \eqref{eq:xi_medium_1} and \eqref{eq:xi_medium_2}, and using the fact that, for $m \geqslant 2$, $m-1 \geqslant m(1-2^{-\epsilon})$ and $(u+2)^m+m \leqslant (u+3)^m$,
\begin{align*}
\Prob(S_m \geqslant u,\ \forall i \in \intervallentff{1}{m} \quad X_i \geqslant  x_n^\epsilon)
 & \leqslant (u+3)^m \exp\bigl[ -q'\bigl( u^{1-\epsilon} + m(1-2^{-\epsilon}) x_n^{\epsilon(1-\epsilon)} \bigr) \bigr] \\
 & \leqslant \exp\bigl[ -(1-\delta)\bigl( u^{1-\epsilon} + m(1-2^{-\epsilon}) x_n^{\epsilon(1-\epsilon)} \bigr) \bigr]
\end{align*}
as soon as
\[
\log(u+3) \leqslant \log(x_n+3) \leqslant \frac{\delta}{3} (1-2^{-\epsilon}) x_n^{\epsilon(1-\epsilon)} ,
\]
i.e.\ for $n \geqslant n(\delta) \geqslant n_2(\delta)$.
\end{proof}

\subsection{Upper bound for the sum of the $\Pi_{n,m}$}

Using the uniform bounds of Lemmas \ref{lem:xi_small} and \ref{lem:xi_medium}, we are able to bound the remaining term $\sum_{m=2}^n \binom{n}{m} \Pi_{n,m}(x_n)$ with an argument mimicing the proof of the upper bound in our unilateral sum-contraction principle.

\begin{lm} \label{lem:maj_pinm}
If $n^{1/2} \ll x_n \ll n^{1/(1+\epsilon)}$, then
\[
\limsup_{n \to \infty} \frac{n}{x_n^2} \log \sum_{m=2}^n \binom{n}{m} \Pi_{n,m}(x_n)
 \leqslant - \frac{1}{2 \sigma^2} .
\]
If $x_n = C n^{1/(1+\epsilon)}$, then
\[
\limsup_{n \to \infty} \frac{n}{x_n^2} \log \sum_{m=2}^n \binom{n}{m} \Pi_{n,m}(x_n)
 \leqslant - J(C) .
\]
\end{lm}

\begin{proof}
Suppose $n^{1/2} \ll x_n \leqslant C n^{1/(1+\epsilon)}$. Fix some integer $r \geqslant 1$. Noticing that
\[
\enstq{(x,y) \in (\R_+)^2}{x + y \geqslant 1} \subset \bigcup_{k=1}^r \enstq{(x,y) \in (\R_+)^2}{x \geqslant \frac{k-1}{r},\ y \geqslant 1-\frac{k}{r}} ,
\]
we have, for all $m \in \intervallentff{2}{n}$,
\begin{align*}
\Pi_{n,m}(x_n)
 & = \Prob(S_n \geqslant x_n , \ \forall i \in \intervallentff{1}{m} \quad X_i \geqslant x_n^\epsilon , \ \forall i \in \intervallentff{m+1}{n} \quad X_i < x_n^\epsilon) \\
 & \leqslant \sum_{k=1}^r \Prob\Bigl( S_{n-m} \geqslant \frac{k-1}{r} x_n , \ \forall i \in \intervallentff{1}{n-m} \quad X_i < x_n^\epsilon \Bigr) \\
 & \hspace{3cm} \Prob\Bigl( S_m \geqslant \Bigl( 1-\frac{k}{r} \Bigr) x_n , \ \forall i \in \intervallentff{1}{m} \quad X_i \geqslant x_n^\epsilon \Bigr) \\
 & \leqslant \sum_{k=1}^r \exp\biggl[ -(1-\delta)\biggl( \frac{((k-1)/r)^2 x_n^2}{2 n \sigma^2} + \Bigl( 1 - \frac{k}{r} \Bigr)^{1-\epsilon} x_n^{1-\epsilon} + m(1-2^{-\epsilon}) x_n^{\epsilon(1-\epsilon)} \biggr) \biggr] ,
\end{align*}
for $n$ large enough, applying Lemmas \ref{lem:xi_small} and \ref{lem:xi_medium}. Hence,
\begin{align*}
\log \sum_{m=2}^n \binom{n}{m} \Pi_{n,m}(x_n)
 & \leqslant \log \sum_{k=1}^r \exp\biggl[ -(1-\delta)\biggl( \frac{((k-1)/r)^2 x_n^2}{2 n \sigma^2} + \Bigl( 1 - \frac{k}{r} \Bigr)^{1-\epsilon} x_n^{1-\epsilon} \biggr) \biggr] \\
 & \hspace{3cm} + \log \sum_{m=2}^n \binom{n}{m} e^{m(1-2^{-\epsilon}) x_n^{\epsilon(1-\epsilon)}}
\end{align*}
where the latter sum is bounded.

\begin{itemize}
\item If $x_n \ll n^{1/(1+\epsilon)}$, then $x_n^2/n \ll x_n^{1-\epsilon}$ and, applying the principle of the largest term (Lemma \ref{lm-magique}), we get
\[
\limsup_{n \to \infty} \frac{n}{x_n^2} \log \sum_{m=2}^n \binom{n}{m} \Pi_{n,m}(x_n)
 \leqslant - (1-\delta) \Bigl( \frac{r-1}{r} \Bigr)^2 \frac{1}{2 \sigma^2} ,
\]
so, letting $r \to \infty$ and $\delta \to 0$,
\[
\limsup_{n \to \infty} \frac{n}{x_n^2} \log \sum_{m=2}^n \binom{n}{m} \Pi_{n,m}(x_n)
 \leqslant - \frac{1}{2 \sigma^2} .
\]
\item If $x_n = C n^{1/(1+\epsilon)}$, then $x_n^2/n = C^2 n^{(1-\epsilon)/(1+\epsilon)} = C^{1+\epsilon} x_n^{1-\epsilon}$ and, applying the principle of the largest term (Lemma \ref{lm-magique}), we get
\[
\limsup_{n \to \infty} \frac{n}{x_n^2} \log \sum_{m=2}^n \binom{n}{m} \Pi_{n,m}(x_n)
 \leqslant - (1-\delta) \min_{k=1}^r \biggl( \frac{((k-1)/r)^2 }{2 \sigma^2} + \frac{1}{C^{1+\epsilon}}\Bigl( 1-\frac{k}{r} \Bigr)^{1-\epsilon} \biggr) ,
\]
so, letting $r \to \infty$ and $\delta \to 0$,
\[
\limsup_{n \to \infty} \frac{n}{x_n^2} \log \sum_{m=2}^n \binom{n}{m} \Pi_{n,m}(x_n)
 \leqslant - \min_{t \in \intervalleff{0}{1}} \biggl( \frac{t^2}{2 \sigma^2} + \frac{(1-t)^{1-\epsilon}}{C^{1+\epsilon}} \biggr)
 = - J(C) . \qedhere
\]
\end{itemize}
\end{proof}

\section{Proof of Theorem \ref{tm-max}} \label{sec:proof_2}
 
To be complete, we mention a short proof of Theorem \ref{tm-max} that we did not find in the literature. Recall that we may assume that $q=1$ without loss of generality (see the beginning of Section \ref{sec:proof_13}). First,
\[
\p(M_n \geqslant x_n) = 1 - (1 - \Prob(X \geqslant x_n))^n \sim n \Prob(X \geqslant x_n) ,
\]
so $x_n^{-1+\varepsilon} \log \Prob(M_n \geqslant x_n) \to -1$ by \eqref{eq:behav_X1}. As for $S_n$, we introduce the following decomposition 
\[
\p(S_n \geqslant x_n) = P_n + R_n
\]
where
\[
P_n \defeq \p(S_n \geqslant x_n,\ \forall i \in \intervallentff{1}{n} \quad X_i < x_n) \quad \text{and} \quad 
R_n \defeq \Prob(S_n \geqslant x_n ,\ \exists i \in \intervallentff{1}{n} \quad X_i \geqslant x_n).
\]

Theorem \ref{tm-max} is a direct consequence of Lemmas \ref{lm-magique}, \ref{lem:limsupRn0}, and \ref{lem:limsupPn0}.

\begin{lm} \label{lem:limsupRn0}
If $x_n \gg n^{1/(1+\epsilon)}$, then
\begin{equation} \label{limsupRn0}
\lim \frac{1}{x_n ^{1-\epsilon}} \log R_n = -1 .
\end{equation}
\end{lm}

\begin{proof}
Notice that
\[
\p(S_{n-1} \geqslant 0) \p(X \geqslant x_n) \leqslant R_n \leqslant n \p(X \geqslant x_n) .
\]
The central limit theorem provides $\Prob(S_{n-1} \geqslant 0) \to 1/2$ and the result follows.
\end{proof}

\begin{lm} \label{lem:limsupPn0}
If $x_n \gg n^{1/(1+\epsilon)}$, then
\begin{equation}
\varlimsup \frac{1}{x_n ^{1-\epsilon}} \log P_n \leqslant -1 . \label{limsupPn0}
\end{equation}
\end{lm}

\begin{proof}
Using the fact that $\indic_{x \geqslant 0} \leqslant e^x$, for all $q' \in \intervalleoo{0}{1}$,
\begin{align}
P_n
 & \leqslant e^{-q' x_n^{1-\epsilon}} \E \bigl[e^{q' x_n^{-\epsilon}X} \indic_{X < x_n}\bigr]^n
 = e^{-q' x_n^{1-\epsilon}} \bigl( \E \bigl[e^{q' x_n^{-\epsilon}X} \indic_{X < x_n^\epsilon}\bigr] + \E \bigl[e^{q' x_n^{-\epsilon}X} \indic_{x_n^\epsilon \leqslant X < x_n}\bigr] \bigr)^n . \label{eq:maxjump_Pn0a}
\end{align}

First, there exists $c > 0$ such that, for all $t \leqslant u$, $e^t \leqslant 1 + t + c t^2$. Therefore,
\begin{align*}
\E \left[e^{q' x_n^{-\epsilon}X} \indic_{X < x_n^\epsilon} \right]
 & \leqslant \E \left[1 + q' x_n^{-\epsilon}X + c (q' x_n^{-\epsilon}X)^2 \right] = 1 + O(x_n ^{-2\epsilon}) .
\end{align*}

Second, integrating by parts,
\begin{align*}
\E \left[e^{q' x_n^{-\epsilon}X} \indic_{x_n^\epsilon \leqslant X < x_n}\right]
 = \int_{x_n^{\epsilon}}^{x_n} e^{q' x_n^{-\epsilon} y} \p(X \in dy)
 = - \bigl[ e^{q' x_n^{-\epsilon} y} \p(X \geqslant y) \bigr]_{x_n^\epsilon}^{x_n} + q' x_n^{-\epsilon} \int_{x_n^{\epsilon}}^{x_n} e^{q' x_n^{-\epsilon} y} \p(X \geqslant y) dy .
\end{align*}
Let $q'' \in \intervalleoo{q'}{1}$. Using \eqref{eq:behav_X1}, for $n$ large enough, we deduce that
\begin{align*}
\E \left[e^{q' x_n^{-\epsilon}X} \indic_{x_n^\epsilon \leqslant X < x_n}\right]
 \leqslant e^{q'} \p(X \geqslant x_n^\epsilon) + q' x_n^{-\epsilon} \int_{x_n^{\epsilon}}^{x_n} e^{q' x_n^{-\epsilon} y - q'' y^{1-\epsilon}} dy .
\end{align*}
The convex function $f_n(y)=q' x_n ^{-\epsilon}y - q''y^{1-\epsilon}$ attains its maximum on $\intervalleff{x_n^\epsilon}{x_n}$ on the boundary. Since $f_n(x_n^{\epsilon}) = q' - q'' x_n^{\epsilon(1-\epsilon)}$, $f_n(x_n)=(q'-1) x_n^{1-\epsilon}$, and $q' \in \intervalleoo{0}{1}$, $f_n(x_n) \leqslant f_n(x_n^{\epsilon})$ for $n$ large enough, whence
\[
\E \left[ e^{q' x_n ^{-\epsilon}X} \indic_{x_n^\epsilon \leqslant X < x_n } \right] \leqslant (1+q' x_n^{1-\epsilon}) e^{q' - q'' x_n^{\epsilon(1-\epsilon)}} = O(x_n ^{-2\epsilon}) .
\]
Consequently, for $x_n \gg n^{1/(1+\epsilon)}$, 
\begin{align*}
\frac{1}{x_n^{1-\epsilon}} \log P_n
 & \leqslant -q' + \frac{n}{x_n^{1-\epsilon}} \log(1 + O(x_n ^{-2\epsilon}))
 = -q'+O\left( \frac{n}{x_n ^{1+\epsilon}} \right),
\end{align*}
and the conclusion follows letting $q'\to 1$.  
\end{proof}

\bibliographystyle{abbrv}
\bibliography{biblio_PGD}

\end{document}

%% file: packmin.tex
	\usepackage{babel}					%Typographie

	\usepackage{amsmath}					%Math de l'AMS
	\usepackage{amssymb}					%Autres symboles mathématiques
	\usepackage{amsthm}					%Théorèmes correctement organisés

	\usepackage{graphicx}				%Insérer des images
	\usepackage{color}					%Couleurs (nécessaires pour les images avec XFig)

	\usepackage{version}					%Commenter (notamment un environnement)

	\usepackage{geometry}				%Marges